\documentclass{amsart}
\usepackage{graphicx}
\usepackage{amsthm,amsfonts,amssymb}
\usepackage{amsmath}
\usepackage{hyperref}
\usepackage[latin1]{inputenc}
\usepackage{latexsym}

\newtheorem{theor}{Theorem}[section]
\newtheorem{lemma}{Lemma}[section]
\newtheorem{proposition}{Proposition}[section]
\newtheorem{corollary}{Corollary}[section]
\newtheorem{defi}{Definition}[section]
\newcommand \be{\begin{eqnarray*}}
\newcommand \ee{\end{eqnarray*}}
\newcommand \ben{\begin{eqnarray}}
\newcommand \een{\end{eqnarray}}

\newcommand{\bigO}[1]{\mathcal O\pa{#1}}

\newcommand{\Pnt}{\mathbb W_n}

\renewcommand{\P}{\mathbb P}

\newcommand{\R}{\mathbb R}

\newcommand{\N}{\mathbb N}
\newcommand{\ac}[1]{\left\{#1\right\}}
\newcommand{\norm}[1]{\left\Vert#1\right\Vert}
\newcommand{\pa}[1]{\left(#1\right)}

\def \build#1#2#3{\mathrel{\mathop{\kern 0pt#1}\limits_{#2}^{#3}}}

\def \esp#1#2{\mathbb E_{#1}\left[#2\right]}
\def \build#1#2#3{\mathrel{\mathop{\kern 0pt#1}\limits_{#2}^{#3}}}

\def\AAk#1#2{\setbox\AAbo=\hbox{#2}\AAdi=\wd\AAbo\kern#1\AAdi{}}%

\newdimen\AAdi%
\newbox\AAbo%

\title[Factors in finite long length words]{Positions of the ranks of factors in certain finite long length
words}

\begin{document}
%\begin{frontmatter}

\author{Elahe Zohoorian Azad}
\address{School of mathematics and computer sciences, Damghan
university, Iran, p.o.Box 36716-41167}
 \email{zohorian@du.ac.ir}

\subjclass{68R15, 60B10, 68Q25}

%\keywords{ Finite random words, factorization of words, rank of
%factor, permutation, uniform distribution.}
\maketitle
%\begin{abstract}
We consider the set of finite random
words $\mathcal A^\star$, with independent letters drawn from a
finite or infinite totally ordered alphabet according to a general
probability distribution. On a specific subset of $\mathcal
A^\star$, considering certain factorization of the words which are
labelled with the ranks, base on the lexicographical order, we
prove that the normalized position of the ranks of factors,
 are uniform, when the length of the word goes to infinity.
%\end{abstract}

%\end{frontmatter}

%MMMMMMMMMMMMMMMMMMMMMMMMMMMMMMMMMMMMMMMMMMMMMMMMMMMMMMMMMMMMMMMMMMMMMMM

\section{Introduction}
\label{intro}
 We consider a general probability distribution
$(p_{i})_{i\geq 1}$ ($p_{i}>0$) on a set $\mathcal A=\ac{a_{1}<
a_{2}<\dots}$ of letters, and we assume, without loss of
generality, that $0<p_{1}<1$. On the corresponding set of words,
$\mathcal A^\star=\bigcup_{n\ge 0}\mathcal A_n$ ($\mathcal A_n$
the set of words with length $n$), considering the word
$w\in\mathcal A_n$ that can be expressed as
$w=a_{\ell_{1}}a_{\ell_{2}}\dots a_{\ell_{n}}$ ($\ell_i$'s are the
positive integer numbers), we define the weight $p(w)$ as
\[
p(w) = p_{\ell_{1}}p_{\ell_{2}}\dots p_{\ell_{n}}.
\]
With the weight $p(.)$, comes a probability measure on $\mathcal
A_n$, $\mathbb P_{n}(\{w\})=p(w)$, and thus a probability measure
on any subset of $\mathcal A_n$. Letting $\mathfrak F_n$ denotes
the $\sigma$-algebra generated by $\mathcal A_n$, the triple of
$(\mathcal A_n,\mathfrak F_n,\mathbb P_{n})$ is the corresponding
probability
space.\\

We recall here some general definitions from \cite{Lothaire} (readers can see also
\cite{Reutenauer,Lothaire1,Lothaire2}). A word $v$ is a \emph{factor} of a word $w$ if
there exists two other words $s$ and $t$, possibly empty, such
that $w = svt$. If $s$ is empty $v$ is a \emph{prefix} (or a right
factor) of $w$ and if $t$ is empty $v$ is a \emph{suffix} (or a
left factor) of $w$.

 A \emph{lexicographic order} on the set of
words $\mathcal A^\star$ is given by a total order on the alphabet
$\mathcal A=\{a_1,a_2,\dots\}$ extended to the words in the
following way: A word $u$, is said to be \emph{smaller} than a
word $v$ if $u$ is a prefix of $v$ or $u=r a_i s$ and $v=r a_j t$
such that $i<j$ and $r,\ s$ and $t$ be some words, possibly empty.

For $w=w_1 \dots w_n$, a word in $\mathcal A_n$ ($n>0$), we define
$\tau w=w_2 \dots w_n w_1$. Then $\langle\tau\rangle=\{Id,\tau,
\dots ,\tau^{n-1}\}$ is the group of cyclic permutations of the
letters of a word with length $n$. The orbit $\langle w\rangle$ of
a word $w$ under $\langle\tau\rangle$ is called a
\textit{necklace}.

A word $w \in \mathcal A_n$ ($n>0$) is called \textit{primitive}
if its necklace $\langle w\rangle$ has exactly $n$ elements. In
other words, a word $w \in \mathcal A_n$ is primitive if it is not
a power of another word in $\mathcal A^*$ (remark that a word $w$
is a power of another word $u$, if $w$ can be written as $w =
uuu\dots u$). Denote by $\mathcal P_n$ the set of primitive words
in $\mathcal A_n$ and by $\mathcal N_n$
its complement.\\

In this article we work on a subset of $\mathcal P_n$ ($n>0$)
containing the primitive words which begin with a run of the own
smallest letter and end with a run of a letter different from the
smallest. We denote this subset by $\mathcal W_n$. For example the
word $a_2 a_3 a_4 a_2 a_3$ is a word in $\mathcal W_5$. The set of
$\mathcal W_n$ contains the words with certain properties that can
be interesting in some applications of combinatorial on words. The
Lyndon words with length $n$ (the words which are strictly smaller
than any their proper suffix), for example, are included in
$\mathcal W_n$. We consider then the probability measure $\mathbb
W_{n}$ on $\mathcal W_n$ (the conditioning probability in the
probability space of $(\mathcal A_n,\mathfrak F_n,\mathbb
P_{n})$):
\begin{equation}
\label{PW} \mathbb W_{n}(\{w\}) = \frac{\mathbb
P_{n}(\{w\})}{\mathbb P_{n}(\mathcal W_n)}.
\end{equation}

Now, we divide the words of $\mathcal W_n$ to the factors that we
call the \emph{blocks of word}, in the following way:
\begin{defi}
\label{defi1} Let $w$ be a word in $\mathcal W_n$ and $a_w$ the
smallest letter of $w$ (remark that $w$ begins with $a_w$). The
blocks of $w$ are the factors of $w$ that begin with a run of
``$a_w$'' and end just before very next run of ``$a_w$''.
\end{defi}

Thus, the blocks of a word are the factors in the form ${a_w}^{k_0}
a_{\ell_{1}}^{k_1} a_{\ell_{2}}^{k_2}\dots a_{\ell_{m}}^{k_m}$
such that $\ a_{\ell_{i}}\neq a_w$ and
$k_0> 0$. \\

Now, concerning the lexicographical order, each block of a word
$w$ in $\mathcal W_n$ can be \emph{ranked}, according to the order
of the word in the necklace of $\langle w\rangle$ which begins by
the mentioned block. For example, in the word
 $w= a c a^2 b a c d b a^2 b a^3
d^2 b$, the blocks and their related ranks below them,
are:$$\build{ \build{a c }{}{}}{4}{}, \build{\build{a^2 b
}{}{}}{3}{},\build{\build{a c d b}{}{}}{5}{}, \build{\build{a^2
b}{}{}}{2}{}, \build{\build{a^3 d^2 b}{}{}}{1}{},
$$
as there is
 the following order between the five words of the necklace
 $\langle w\rangle$, which begin by the blocks of $w$:
\begin{eqnarray*}
a^3 d^2 ba c a^2 b a c d b a^2 b & \leq & a^2 b a^3 d^2 ba c a^2 b
a c d b \ \ \leq \ \  a^2 b a c d b a^2 b a^3 d^2 b a c
\\
&\leq & a c a^2 b a c d b a^2 b a^3 d^2 b\ \ \leq \ \ a c d b a^2 b a^3
d^2 ba c a^2 b.
\end{eqnarray*}

In this work, we are interested in the limiting distribution of
the positions of ranks of a random word of $\mathcal W_n$, which
seems to be uniform. It is trivial that if the uniform random
permutation of the blocks, causes the uniform displacements of the
ranks, the purpose would be entailed. But in general, it is not
true because of the existence of the equal blocks. In fact, it is
well possible that certain permutations of ranks are not produced.
For example, for the word $w=a^2b, a^2 b, ab, ab$ with the
respective ranks $1,2,4,3$ of its blocks, the permutation
$1,2,3,4$ of the ranks is not produced by any permutation of the
blocks of $w$, no more any permutation of the cyclic permutation
of $1,2,3,4$ as $4,1,2,3$. Moreover, certain permutations of the
blocks of a word, like the permutation $a^2b, ab, a^2 b, ab$ of
the blocks of $w$, produce the non-primitive words for which the
ranks of the blocks are not defined. Nevertheless, by definition
of the ranks, it is trivial that the cyclic permutations of the
blocks of a word causes the cyclic permutations of the related
ranks. On the other hand, in any orbit of the cyclic permutation
of the ranks, the ranks are uniformly distributed on all
positions; the fact that is the key of the prove.\\

 Marchand \& Zohoorian in \cite[Section 6]{Zoh},
for a random word of $\{a,b\}^n$ when $n$ is sufficiently large,
demonstrate hardly that the uniform permutation of the blocks of
the word, barring the block with rank 1, entails the uniform
displacement of the rank 2 on all possible positions. However, our
result may be applied also in the analyze of the height of the
labelled binary trees emerged by the successive iterations of the
standard factorization of the Lyndon words (an introduction on
Lyndon trees is given by Marchand \& Zohoorian in \cite[Section
1]{Zoh} or by Bassino \emph{et al.} in \cite{Bassino}). In fact,
the height of Lyndon trees has a direct relation with the
positions of the ranks in the root word. The study on the
structure of Lyndon trees is a work in progress. The main theorem
that we will prove in this article is:

\begin{theor} \label{main} The positions of ranks in
a random word of $\mathcal W_n$, divided by
$\frac{p_1(1-p_1)n}{2}$, converge in law, when $n$ goes to
infinity, to $$ \mu(dx) =
 \mathbf{1}_{[0,1]}(x)dx,
$$
where $dx$ denotes the Lebesgue measure on $\R$.
\end{theor}

 \section{Number of the blocks of a word}
\label{bloc} To begin, we concentrate on the number of the blocks
of a random word in $\mathcal W_n$. It is evident that the number
of the blocks of a word in $\mathcal W_n$ is equal to the number
of the runs of its smallest letter. Now, we define the following
function on $\mathcal P_n$ that carries the words of $\mathcal
P_n$ to its subset $\mathcal W_n$:

\begin{defi}
\label{pi} Let $\phi$ denotes the function on $\mathcal P_n$ that
brings any word $w$ of $\mathcal P_n$ to itself if $w$ is a word
in $\mathcal W_n$, and otherwise to the word of $\langle w\rangle$
which begins by the last run of the smallest letter of $w$ (see
the following example).
\end{defi}
\vspace{0.5cm}
 \textbf{Example.}\begin{itemize}
\item If $w=aba^2 cb^2 a^2$, then $\phi(w)=a^3 b a^2 c b^2$. \item
If $A=\{ab a^2 c,acbab \}\subset \mathcal W_5$ then
${{\phi}^{-1}}(A)=\{ab a^2 c,b a^2 ca\}\cup\{a cbab,cbaba ,baba
c\}$.
\end{itemize}
\vspace{0.5cm} We remark that $\phi$ is a surjective map on $\mathcal
P_n$ to $\mathcal W_n$ and the inverse
 image of $w\in\mathcal W_n$ is a subset of $\mathcal P_n$
 whose cardinality is equal to the length of the first block
 of $w$.\\

  The next lemma allows to transfer
results from random words of $\mathcal A_n$ to random words of
$\mathcal W_n$. For this mean, we set
\[\norm{p}_{\alpha}=\pa{\sum_{i}p_{i}^{\alpha}}^{1/\alpha},\] for $\alpha\ge
1$.\\

\textbf{Notation.} If $a$ and $b$ are two functions from $\N$ into
$\R$, $a(n)=\bigO{b(n)}  \Leftrightarrow  \exists c>0, \; \forall
n \in \N, \; |a(n)| \le c |b(n)|$.\\

\begin{lemma}
\label{usarg} For $A \subset \mathcal W_n$, we have:
 $$\mid \Pnt (A)- \P_{n}({\phi}^{-1} (A)) \mid = \bigO{\norm{p}_{2}^{n}}.$$
\end{lemma}
Note that $\norm{p}_{1}=1$, and that, under the assumption
$\ac{0<p_{1}<1}$, $\norm{p}_{\alpha}$ is strictly decreasing in
$\alpha$. Other well known inequalities include  $\norm{p}_{2}\le
\sqrt{\max p_{i}}$.\\

The lines of the proof of this lemma are exactly the same as in
the proof of Lemma 2.1 of \cite{Ch}, but here we work on the set
$\mathcal W_n$, instead of the set of Lyndon words, there.
So we remove the proof, referring
the reader to \cite[Lemma 2.1]{Ch}. \\

Let now for any word $w$ in $\mathcal A_n$ we denote by
$N_n^{a_1}(w)$ the number of runs of the letter $a_1$ in the word
$w$. We have then the following lemma:

\begin{lemma}
\label{N} (Number of runs of the letter $a_1$).
\begin{eqnarray*}
\Pnt\pa{N_n^{a_1}< \frac{p_1(1- p_1)}{2}\ n } = \bigO{n^{-1}}.
\end{eqnarray*}
\end{lemma}
We remove the proof of this lemma, again inviting the
reader to see \cite[Lemma 2.3]{Ch}. \\

By Lemma \ref{N}, one sees that the number of the blocks of a
word with length $n$, when $n$ is sufficiently large, is of the
order ${n}$ with a high probability as a word with length $n$ has
at most $\frac n 2$ blocks (the case where the smallest letter
repeats alternatively).

\section{Displacements of ranks }
\label{perm} We consider, at first, some notations and
definitions. As a consequence of Definition \ref{defi1}, any word
$w\in \mathcal W_n$ can be decomposed uniquely as
\[
w=B_1(w),B_2(w),\dots ,B_{N_n(w)}(w),
\]
in which $B_i(w)$'s stand for the blocks of $w$ and $N_n(w)$
denotes the number of the blocks. We denote the respective ranks
of the blocks of $w$ by
\[
r(w)=r_1(w),r_2(w),\dots,r_{N_n(w)}(w),
\]
called, briefly, the rank of $w$. Obviously,
$r(w)$ is a permutation of $1,2,\dots,N_n(w)$.

\begin{defi}
\label{blocorbit} For any $w=B_1,B_2,\dots ,B_{N}\in \mathcal W_n$
in which $B_i$'s are the blocks of $w$, we define $\beta
w=B_2,\dots ,B_{N},B_1$. Then $\langle \beta\rangle=\{Id,\beta,
\dots ,\beta^{{N}-1}\}$ is the group of cyclic permutations of the
blocks of a word in $\mathcal W_n$. We call the orbit $\langle
w\rangle_\beta$ of the blocks of $w$ under $\langle \beta\rangle$,
the \textit{block orbit} of $w$.
\end{defi}

\begin{proposition}
\label{invariant} The ranks of the blocks of a word in $\mathcal
W_n$ are invariant under the cyclic permutation of the blocks.
\end{proposition}
\begin{proof} By definition of the ranks, it is evident that
 the ranks are permuted cyclically as the blocks are permuted.
\end{proof}

The following corollary is an immediate result of the above
proposition and the proof is left as it is evident.

\begin{corollary}
\label{uniposition} In any block orbit, the ranks are distributed
uniformly on all positions. In other words, if the number of the
words in a block orbit is $N$, the probability that the $i$-th
rank, $i=1,\dots,N,$ be in the position $k,\ k=1,\dots,N,$ is
equal to $\frac{1}{N}$.
\end{corollary}

Let $\mathfrak{S}_{n}$ denote the set of permutations of $\{1,
\dots, n\}$. For $w \in \mathcal W_n$, and
$\sigma\in\mathfrak{S}_{N_n(w)}$, we set
$$\sigma.w=B_{\sigma(1)}(w),\dots,B_{\sigma(N_n(w))}(w).$$
Conditioning then $\sigma.w\in \mathcal W_n$, the rank of
$\sigma.w$, $r(\sigma.w)$, is also a permutation in
$\mathfrak{S}_{N_n(w)}$. We set $C(w)=\{\sigma.w\ :\ \sigma \in
\mathfrak{S}_{N_n(w)}\ \& \  \sigma.w\in \mathcal W_n \}$ and
$\mathcal C_n=\ac{C(w)\ :\ w\in \mathcal W_n}$. Let $\mathfrak
C_n$ denote the $\sigma$-algebra generated by $\mathcal C_n$. In
the following proposition we see that $\mathcal W_n$ is parted to
$C(w)$s:

\begin{proposition}
\label{equivclass}$\mathcal C_n$ is a partition of $\mathcal W_n$.
\end{proposition}

\begin{proof}Assume that $w\in\mathcal W_n$, and $w'\in
C(w)$: then $w'\in \mathcal W_n$ and $w'$ has the same multiset
 of blocks as $w$ (it has the same blocks, with the same multiplicity).
As a consequence, for $w, w' \in \mathcal W_n$, either
$C(w)=C(w')$ or $C(w)\cap C(w')=\emptyset$.
\end{proof}

In the following proposition, which is the key proposition of this
result, we see that the {block orbits} divide $C(w)$:

\begin{proposition}
\label{BO} For any $w\in \mathcal W_n$, $C(w)$ is parted to the {block
orbits}.
\end{proposition}
\begin{proof}As $C(w)$
contains the primitive words produced by permutations of the
blocks of $w$, it is sufficient to verify that all cyclic
permutations of the blocks of a word in $\mathcal W_n$ produce the
primitive words. Suppose now that a cyclic
permutation of the blocks of a word $v=B_1,B_2,\dots,B_{N}\in
\mathcal W_n$, for example $\beta^i
v=B_{i+1},B_{i+2},\dots,B_{N},B_1,\dots,B_{i}$ ($1\leq i\leq N$),
is a non-primitive word in the form $u^r$ for $u\in\mathcal A^*$
and $2\leq r\leq n$. As $u$ is a prefix and a
suffix of $\beta^i v$ and $\beta^i v$ is a word in $\mathcal W_n$,
$u$ is a word which begins by its smallest letter and end with a
letter different from the smallest letter. Therefore $u$ can be
factorized in the blocks and so has exactly the same blocks as
the first blocks of $\beta^i v$. Consequently, the blocks of $\beta^i v$
are periodically equals. That is, if $u$ has $j$ number of the blocks, they are
equals to the first $j$, second $j$, ... and last $j$ blocks of
$\beta^i v$. But this equalities, in $\beta^i v$, entail that for any
$k=1,2,\dots,N-j$, $B_k=B_{k+j}$ and that
all cyclic permutations of the blocks of $\beta^i v$, especially
$v$, will be the words which can be part to $\frac N j$ equal factors.
Therefore, any cyclic permutation of the blocks of a word in $\mathcal W_n$,
can not be a non-primitive word.
\end{proof}

%MMMMMMMMMMMMMMMMMMMMMMMMMMMMMMMMMMMMMMMMMMMMMMMMMMMMMMMMMMM
%MMMMMMMMMMMMMMMMMMMMMMMMMMMMMMMMMMMMMMMMMMMMMMMMMMMMMMMMMM

% \vspace{0.3cm}
\section{Proof of Theorem \ref{main}}
\label{proof} Let $\bar{B}(w)=\pa{B_i(w)}_{i\ge 0}$ be the
sequence of blocks of $w$, ended by an infinite sequence of empty
words, and let $\bar{r}(w)=\pa{r_i(w)}_{i\ge 0}$ be the
corresponding sequence of ranks.
\begin{lemma}
\label{condition2} The weight $p(.)$, $\bar{B}$, $\bar{r}$ and
$N_n$ are $\mathfrak C_n$-measurable, and
$$
\mathbb W_n = \sum_{C\in\mathcal C_n}\frac{\mathrm{Card}(C)\
p(C)}{\mathbb P_n(\mathcal W_n)}\ \mathbb U_C.
$$
Given that $w\in C$, the positions of the ranks of $w$,
$(r_{i}(w))_{1\le i\le N_{n}(C)}$, are distributed uniformly on
$\ac{1,2,\dots,N_n(C)}$.
\end{lemma}
\begin{proof} The weight $p(w)$ depends only on the number of
letters $a_{1}$, $a_{2}$, \dots that $w$ contains, not on the
order of the letters in $w$, so that $p(.)$ is constant on each
$C\in\mathcal C_n$: thus, under $\mathbb W_n$, the conditional
distribution of $w$ given that $w\in C$ is $\mathbb U_C$. As a
consequence of Proposition \ref{equivclass}, the relation in Lemma
\ref{condition2} is just the disintegration of $\mathbb W_n$
according to its conditional distributions given $\mathfrak C_n$.
Finally, $\mathbb U_{C(w)}$ is the image of the uniform
probability on $[\mathfrak{S}_{N_n(w)} \times
\ac{w}]\bigcap\mathcal W_n$. Thus, by Propositions
\ref{invariant}, \ref{BO} and Corollary \ref{uniposition}, under
$\mathbb U_{C(w)}$, the positions of the ranks are distributed
uniformly on $\ac{1,2,\dots,N_n(C)}$. It follows that, under
$\mathbb W_n$, the conditional distribution of the positions of
the ranks given $C(w)$, or given $N_n$, is uniform too.
\end{proof}

We can see, by Lemma \ref{N}, the probability that the number of
blocks of a random word of $\mathcal W_n$, $N_n$, be of order $n$,
increases when $n$ increases. Now, let the position of the $i$-th
rank, denoted by $I_n$, is distributed uniformly on
$\ac{1,2,\dots,N_n(C)}$. We put ${\mathcal
I_n}\equiv{\frac{2}{p_1(1- p_1)n}}{ I_n}$, the normalized position
of the $i$-th rank by $\frac{p_1(1- p_1)}{2}n$. We shall see that
${\mathcal I_n}$ is approximately uniform, for its distribution is
close to the \emph{uniform distribution} on $[0,1]$, that we note
$\mathbb U$ in the rest of paper. This proximity is understood
with respect to the $\mathcal L_2$-\textit{Wasserstein metric}
$W_2(.,.)$. The $\mathcal L_2$-Wasserstein metric $W_2(.,.)$ is
defined by
\begin{eqnarray}
\label{metric}  W_2(\mu,\nu) &=& \inf_{{\scriptstyle\mathcal{L}(X)
=\mu}\atop{\scriptstyle\mathcal{L}(Y)=\nu}}\esp{}{\left\|X-Y\right\|_2^2}^{1/2},
\end{eqnarray}
in which $\mu$ and $\nu$ are probability distributions on $\mathbb
R^d$, and $\norm{.}_2$ denotes the Euclidean norm on $\mathbb
R^d$. In this paper, we consider essentially the case $d=1$. We
mention that convergence of $\mathcal L(X_n)$ to $\mathcal L(X)$
with respect to $W_2(.,.)$ entails convergence of $X_n$ to $X$ in
distribution, and we refer to \cite{Rachev} for an extensive
treatment of Wasserstein metrics. In what follows, we shall
improperly refer to the convergence of $X_n$ to $X$ with respect
to $W_2(.,.)$, meaning the convergence of their distributions. The
main reason for the asymptotic uniformity of $\mathcal I_n$ is a
form of convergence of the empirical distribution function
$G_n(t)$ to $t$, in the notations of \cite[Ch. 3.1, p.85, display
(3)]{Shorack}:

\begin{lemma}
\label{larcin} Consider a partition of $[0,1)$ into $n$ intervals
$\left[\frac{i-1}n,\frac{i}n\right);i=1,\dots,n$. For $\mathfrak
r$ a random cyclic permutation of the class of all cyclic
permutations of the intervals, we have
\[W_2(\mathfrak
r(i),\mathbb U)\le\sqrt{1/n}.\]
\end{lemma}

This Lemma is also a specific case of \cite[Lemma 6.3]{Zoh} where
one can see for a proof.\\

Theorem \ref{main} is demonstrated when the following proposition
is proved:
\begin{proposition}
\label{distance} Let $\nu_n$ be the distribution of $\mathcal I_n$
under $(\mathcal W_n,\mathbb W_n)$. We have then:
\[
 W_2\pa{\nu_n,\mathbb U} =\bigO{\sqrt{1/n}}.
\]
As a consequence, under $\mathbb W_n$, the moments of $\mathcal
I_n$ converge to the corresponding moments of $\mathbb U$.
\end{proposition}

\begin{proof}
With the notations of Lemma \ref{condition2}, for $C\in\mathcal
C_n$, let $\nu_C$ denote the image of $\mathbb U_C$ by $\mathcal
I_n$, so that
\begin{equation}
\label{desint1} \nu_n= \sum_{C\in\mathcal
C_n}\frac{\mathrm{Card}(C)\ p(C)}{\mathbb P_n(\mathcal W_n)}\
\nu_C.
\end{equation}
Consider the blocks $b_1\le\dots\le b_{N_n}$ of a word
$w\in\mathcal W_n$, sorted in increasing lexicographic order.
There exists at least one permutation $\tau\in\mathfrak S_{N_n}$
such that $w=b_{\tau(1)}\dots b_{\tau(N_n)}$. Let $\mathfrak
R_{N_n}$ denote the set of permutations $\delta$ such that
$\omega=\delta^{-1}.\tau^{-1}.w$ be an element of $C(w)$. Then,
for $\sigma$ is a random uniform element of $\mathfrak R_{N_n}$,
 $\omega=\sigma^{-1}.\tau^{-1}.w$ is a random uniform
element of $C(w)$. Set $\Upsilon(\sigma)=\mathcal
I_n(\sigma^{-1}.\tau^{-1}.w)$. Then, the distribution of
$\Upsilon$ is $\nu_{C(w)}$ by Corollary \ref{uniposition} and
Proposition \ref{BO}. Thus, by a straightforward extension of
Lemmas \ref{larcin} and \ref{N},
\begin{equation}
\label{desint2} W_2\pa{\nu_{C(w)},\mathbb U}\leq \esp{}{(\Upsilon-
U)^2\mid N_n^{a_1}> \frac{p_1(1- p_1)}{2}\ n }^{1/2} \le
\sqrt{\frac{2}{p_1(1- p_1)n}}.
\end{equation}
Finally, joining (\ref{desint1}) and (\ref{desint2}) we obtain
\[
 W_2\pa{\nu_n,\mathbb U} =\bigO{\sqrt{1/{n}}}.
\]

and since $0\le \mathcal I_n\le 1$, convergence of moments
follows.
\end{proof}

%MMMMMMMMMMMMMMMMMMMMMMMMMMMMMMMMMMMMMMMMMMMMMMMMMMMMMMMMMMM

\end{document}